\documentclass[11pt]{article}
\usepackage[T1]{fontenc}
\usepackage[utf8]{inputenc}
\usepackage{lmodern}
\usepackage[a4paper,margin=28mm]{geometry}
\usepackage{amsmath,amssymb,amsthm,mathtools}
\usepackage{booktabs,array,microtype}
\usepackage{xcolor}
\usepackage[colorlinks=true,linkcolor=blue!50!black,citecolor=blue!50!black,urlcolor=blue!50!black]{hyperref}
\usepackage[nameinlink,noabbrev]{cleveref}
\numberwithin{equation}{section}
\newtheorem{theorem}{Theorem}[section]
\newtheorem{lemma}[theorem]{Lemma}
\newtheorem{proposition}[theorem]{Proposition}
\newtheorem{corollary}[theorem]{Corollary}
\theoremstyle{definition}
\newtheorem{definition}[theorem]{Definition}
\theoremstyle{remark}

\newcommand{\E}{\mathbb E}
\newcommand{\Pbb}{\mathbb P}
\newcommand{\R}{\mathbb R}
\newcommand{\eps}{\varepsilon}

\newcommand{\kopt}{\kappa_{\mathrm{opt}}}
\newcommand{\kan}{\kappa_{\mathrm{an}}}
\newcommand{\tauopt}{\tau_{\mathrm{opt}}}
\newcommand{\tauan}{\tau_{\mathrm{an}}}
\allowdisplaybreaks[2]
\title{Moment Constants for Brownian Affine Approximation}
\author{Yue Chen\thanks{Email:
\href{mailto:chenyue20020601@gmail.com}{\texttt{chenyue20020601@gmail.com}}}}
\date{}
\begin{document}
\maketitle
\begin{abstract}
We evaluate moment constants for the times during which Brownian motion
admits uniform approximation by an unrestricted affine function or by a
line through the origin. The first four moments are rational linear
combinations of $1$ and $\zeta(2j+1)/\pi^{2j}$, $1\le j\le4$;
the means are $40\zeta(3)/\pi^2-4/3$ and $21\zeta(3)/\pi^2$,
respectively. Slope integration and a common contour argument give the
unrestricted moments, while time inversion relates the anchored law to
the reciprocal of Kingman's absorption time. For both models we also
give exponential-integral series for the standardized absolute third
central moments, with explicit truncation bounds.
\end{abstract}
\section{Introduction and main results}
\label{intro:section}

Throughout, $W=(W_t)_{t\ge0}$ denotes standard real Brownian motion,
with $W_0=0$ and $\E W_t^2=t$. For $f\in C[0,T]$ with $f(0)=0$, define
\begin{align}
 D_{\mathrm{opt}}(f,T)
  &=\inf_{a,b\in\R}\sup_{0\le t\le T}|f(t)-at-b|,
  \label{def:optimal}\\
 D_{\mathrm{an}}(f,T)
  &=\inf_{a\in\R}\sup_{0\le t\le T}|f(t)-at|.
 \label{def:anchored}
\end{align}
For $T>0$, both infima are attained. Indeed, bounds on the residuals
at $0$ and $T$ bound all fitting coefficients, so every nonempty error
sublevel set in coefficient space is compact.

\begin{definition}
\label{def:lifetimes}
For $A\in\{\mathrm{opt},\mathrm{an}\}$, let
\begin{equation}
 \tau_A=\inf\{t>0:D_A(W,t)>1\},\qquad
 m_{q,A}=\E\tau_A^q\quad(q>0),\qquad \kappa_A=m_{1,A}.
 \label{def:times}
\end{equation}
The associated dispersion constants are
\begin{equation}
 V_A=m_{2,A}-\kappa_A^2,\qquad
 \alpha_A=\frac{V_A}{\kappa_A^2},\qquad
 \beta_A=\frac{\E|\tau_A-\kappa_A|^3}{V_A^{3/2}}.
 \label{def:dispersion}
\end{equation}
We write $S_A(t)=\Pbb(\tau_A>t)$ for the survival function and
$\zeta(s)=\sum_{n\ge1}n^{-s}$ for $s>1$.
\end{definition}

The unrestricted time is $\tau_{1,0,\infty}$ in the notation of
Creutzig, M\"uller-Gronbach, and Ritter~\cite{CMR2007}; their approximation
constant satisfies $c_{1,0,\infty}=\kopt^{-1/2}$.

We evaluate the first four moments of both approximation times.
For unrestricted approximation, integration over slopes and known
Brownian extremum distributions \cite{SchehrLeDoussal2010,Mori2020}
give an explicit survival series and reduce the four moments to a
common application of the residue theorem, including the nonzero
boundary term in the mean. For anchored approximation, the classical
wedge probability \cite{YcartDrouilhet2016} identifies the lifetime
with a scaled reciprocal of the Kingman absorption time
\cite{MoehlePitters2015}. All positive integer moments of the anchored
lifetime are evaluated as finite combinations of odd zeta values.
For both models, we also
give exact exponential-integral series for the standardized absolute
third moments, with explicit truncation bounds. Proofs of the
distribution identities needed for these evaluations are included.

\begin{theorem}[Moment constants]
\label{const:evaluations}
For unrestricted affine approximation,
\begin{align}
 \kopt&=\frac{40\zeta(3)}{\pi^2}-\frac43,
 \notag\\
 m_{2,\mathrm{opt}}&=\frac{2240\zeta(5)}{\pi^4}
                       -\frac{224\zeta(3)}{3\pi^2},
 \label{const:optimal-moments}\\
 m_{3,\mathrm{opt}}&=\frac{120960\zeta(7)}{\pi^6}
                 -\frac{5440\zeta(5)}{\pi^4}
                 +\frac{64\zeta(3)}{3\pi^2},
 \notag\\
 m_{4,\mathrm{opt}}&=\frac{7096320\zeta(9)}{\pi^8}
                 -\frac{365568\zeta(7)}{\pi^6}
                 +\frac{3072\zeta(5)}{\pi^4}.
 \notag
\end{align}
For approximation by a line through the origin,
\begin{align}
 \kan&=\frac{21\zeta(3)}{\pi^2},
 \notag\\
 m_{2,\mathrm{an}}&=\frac{930\zeta(5)}{\pi^4}
                         -\frac{14\zeta(3)}{\pi^2},
 \label{const:anchored-moments}\\
 m_{3,\mathrm{an}}&=\frac{40005\zeta(7)}{\pi^6}
                         -\frac{930\zeta(5)}{\pi^4},
 \notag\\
 m_{4,\mathrm{an}}&=\frac{1931580\zeta(9)}{\pi^8}
                         -\frac{53340\zeta(7)}{\pi^6}
                         +\frac{124\zeta(5)}{\pi^4}.
 \notag
\end{align}
The corresponding variances are
\begin{align}
 V_{\mathrm{opt}}
  &=\frac{2240\zeta(5)-1600\zeta(3)^2}{\pi^4}
      +\frac{32\zeta(3)}{\pi^2}-\frac{16}{9},
 \label{const:optimal-variance}\\
 V_{\mathrm{an}}
  &=\frac{930\zeta(5)-441\zeta(3)^2}{\pi^4}
      -\frac{14\zeta(3)}{\pi^2}.
 \label{const:anchored-variance}
\end{align}
In both cases, $\alpha_A=V_A/\kappa_A^2$ and
$\kappa_A/\alpha_A=\kappa_A^3/V_A$.
\end{theorem}

The raw moments in \eqref{const:optimal-moments} and
\eqref{const:anchored-moments} are evaluated in
Sections~\ref{opt:section} and~\ref{an:section}, respectively.
The variance formulas follow by subtracting the squared means.

The fourth central moments are obtained without further integration:
\begin{equation}
 \mu_{4,A}:=\E(\tau_A-\kappa_A)^4
 =m_{4,A}-4\kappa_A m_{3,A}+6\kappa_A^2m_{2,A}-3\kappa_A^4.
 \label{const:fourth-central}
\end{equation}
\begin{table}[htbp]
\centering
\caption{Values obtained from the exact formulas and convergent expressions.}
\label{tab:constants}
\begin{tabular}{lrr}
\toprule
 & Unrestricted affine & Through the origin \\
\midrule
$\kappa$ & 3.538419796 & 2.557670393 \\
$m_2$ & 14.751044357 & 8.194812603 \\
$m_3$ & 71.557443239 & 32.059160830 \\
$m_4$ & 398.663120444 & 149.353361971 \\
$\mu_4$ & 23.714827828 & 14.632593213 \\
$V$ & 2.230629704 & 1.653134764 \\
$\alpha$ & 0.178159411 & 0.252708067 \\
$\kappa/\alpha$ & 19.860976021 & 10.121047654 \\
$\beta$ & 1.860139680 & 1.942855459 \\
\bottomrule
\end{tabular}
\end{table}

The standardized absolute third moments in Table~\ref{tab:constants}
are evaluated by the exponential-integral series in
Theorem~\ref{abs:series}; their truncation bounds are given in
Proposition~\ref{abs:truncation}.

Section~\ref{def:section} collects the common scaling and moment identities.
Sections~\ref{opt:section} and~\ref{an:section} derive the distribution
representations and evaluate the raw moments for the two models.
Section~\ref{abs:section} treats their absolute central moments together.
The boundary-variation proof of the ordered-extremum density is given
in Appendix~\ref{app:ordered-density}.

\section{Preliminaries}
\label{def:section}

\begin{lemma}[Order and scaling]
\label{def:scaling}
The error functions are continuous and nondecreasing in $T$. On the same
Brownian path,
\begin{equation}
 \tauan\le\tauopt.
 \label{def:order}
\end{equation}
Moreover,
\begin{equation}
 \tau_A\overset d=D_A(W,1)^{-2}.
 \label{def:inverse-error}
\end{equation}
For $\eps,\sigma>0$ and $d\in\R$, the first time that the corresponding
error of $dt+\sigma W_t$ exceeds $\eps$ has law
$(\eps^2/\sigma^2)\tau_A$.
\end{lemma}

\begin{proof}
Restriction of the interval gives monotonicity. Reparametrization by
$t=Ts$ gives $D_A(f,T)=D_A(f(T\,\cdot),1)$ for each rule.
As functions of the path in the uniform norm, $D_{\mathrm{opt}}$ and
$D_{\mathrm{an}}$ are 1-Lipschitz.
This proves continuity in $T$, including at zero. Optimizing over the
intercept gives $D_{\mathrm{opt}}\le D_{\mathrm{an}}$, proving
\eqref{def:order}.

Continuity and monotonicity give the pathwise equality
$\{\tau_A\ge t\}=\{D_A(W,t)\le1\}$. Brownian scaling gives
$D_A(W,t)\overset d=\sqrt t\,D_A(W,1)$, which proves
\eqref{def:inverse-error} by comparison of closed tails.
Each error is invariant under addition of a linear function and homogeneous
under multiplication of the path by a positive constant. Combining these
properties with Brownian scaling proves the last assertion.
\end{proof}

\begin{lemma}[Finiteness and nondegeneracy]
\label{def:finiteness}
Each $\tau_A$ is positive almost surely, has an exponential moment in a
neighborhood of zero, and is nondegenerate. In particular, all the constants
in \eqref{def:dispersion} are finite, and $V_A>0$.
\end{lemma}

\begin{proof}
Since $D_{\mathrm{an}}(W,t)\le\sup_{s\le t}|W_s|$, continuity of $W$
at zero and \eqref{def:order} imply that both times are positive.
For $j\ge0$, put
\[
 Z_j=\frac12\left(W_{16j+8}
             -\frac{W_{16j}+W_{16j+16}}2\right).
\]
These variables are independent standard normals: the expression uses
disjoint pairs of increments of length 8. If one affine function fits the
three displayed points with error at most 1, cancellation of its affine
terms implies $|Z_j|\le1$. Hence, with $p=\Pbb(|Z_0|>1)>0$,
\begin{equation}
 \Pbb(\tauopt>16n)\le(1-p)^n\qquad(n\ge0).
 \label{def:exponential-tail}
\end{equation}
By \eqref{def:order}, the same bound holds for every $\tau_A$, proving
the exponential-moment assertion.

For every $t>0$, the event $\sup_{s\le t}|W_s|<1/2$ has positive
probability and implies $D_{\mathrm{an}}(W,t)<1$, hence $\tauan>t$.
Positivity of this Brownian small-ball probability follows, for example,
from positivity of the killed heat kernel on $(-1/2,1/2)$.
On the other hand, $W_{t/2}-W_t/2$ is a nondegenerate normal variable;
if its absolute value exceeds 2, then $D_{\mathrm{opt}}(W,t)>1$.
Thus every $\tau_A$ has positive probability of being smaller than any
given positive time and positive probability of exceeding it. This proves
nondegeneracy and $V_A>0$. The third absolute moments are finite by
the exponential tail.
\end{proof}

\begin{lemma}[Moment identities]
\label{def:absolute-third}
For $q>0$,
\begin{equation}
 m_{q,A}=q\int_0^\infty t^{q-1}S_A(t)\,dt.
 \label{def:tail-moment}
\end{equation}
If
$\gamma_A=(m_{3,A}-3\kappa_A m_{2,A}+2\kappa_A^3)/V_A^{3/2}$,
then
\begin{align}
 \beta_A
 &=\gamma_A+\frac{2\E(\kappa_A-\tau_A)_+^3}{V_A^{3/2}}
 \label{def:beta-positive}\\
 &=\frac{m_{3,A}-3\kappa_A m_{2,A}+4\kappa_A^3
       -6\displaystyle\int_0^{\kappa_A}(\kappa_A-t)^2S_A(t)\,dt}
       {V_A^{3/2}}.
 \label{def:beta-survival}
\end{align}
Here $\gamma_A$ is the signed skewness, whereas $\beta_A$ uses the absolute
third central moment.
\end{lemma}

\begin{proof}
Writing $x^q=q\int_0^x t^{q-1}\,dt$ and applying Tonelli's theorem proves
\eqref{def:tail-moment}. The identity $|x|^3=x^3+2(-x)_+^3$ and expansion
of $\E(\tau_A-\kappa_A)^3$ give \eqref{def:beta-positive}. A second
application of Tonelli gives
\[
 \E(\kappa_A-\tau_A)_+^3
 =3\int_0^{\kappa_A}(\kappa_A-t)^2\Pbb(\tau_A\le t)\,dt
 =\kappa_A^3-3\int_0^{\kappa_A}(\kappa_A-t)^2S_A(t)\,dt.
\]
Substitution proves \eqref{def:beta-survival}.
\end{proof}

\section{Optimal affine approximation}\label{opt:section}

For $g\in C[0,1]$, write
\begin{align*}
 F(g)&=\inf_{a,b\in\R}\sup_{0\leq t\leq1}|g(t)-at-b|,\\
 R_g(a)&=\max_{0\leq t\leq1}(g(t)-at)
          -\min_{0\leq t\leq1}(g(t)-at),\qquad
 \rho(g)=\min_{a\in\R}R_g(a).
\end{align*}
For a fixed slope, the optimal intercept is the midpoint between the
maximum and minimum of $g(t)-at$; hence $\rho(g)=2F(g)$.
Throughout this section, $m_q=m_{q,\mathrm{opt}}$.

\subsection{Integration over slopes}\label{opt:deterministic-subsection}

\begin{lemma}[Slope integration]\label{opt:coarea-lemma}
Suppose $g$ is continuous and is not affine.  Then $\rho(g)>0$, and for any
nonnegative Borel function $h$ on $(0,\infty)$,
\begin{equation}\label{opt:coarea-general}
 \int_{\R}|R_g'(a)|h(R_g(a))\,da
       =2\int_{\rho(g)}^\infty h(r)\,dr.
\end{equation}
In particular, for $q>0$,
\begin{equation}\label{opt:coarea-power}
 F(g)^{-2q}=2^{2q}q\int_{\R}
       \frac{|R_g'(a)|}{R_g(a)^{2q+1}}\,da.
\end{equation}
At almost every slope $a$, the maximum and minimum of $g(t)-at$ are
attained at unique times $T_+(g-a\,\cdot)$ and $T_-(g-a\,\cdot)$, and
\begin{equation}\label{opt:range-derivative}
 R_g'(a)=T_-(g-a\,\cdot)-T_+(g-a\,\cdot).
\end{equation}
\end{lemma}

\begin{proof}
The function $R_g$ is convex: each of its two summands is a supremum of
affine functions of $a$.  It is locally Lipschitz and satisfies
\[
 R_g(a)\geq|g(1)-g(0)-a|.
\]
It is therefore coercive and has a nonempty compact interval of
minimizers.  Its minimum is positive, since a zero range would make $g$
affine.  To the left of the minimizer interval $R_g$ is strictly decreasing,
and to its right it is strictly increasing; the possible flat minimum
has derivative zero.  The one-dimensional change-of-variables formula
on these two monotone branches gives \eqref{opt:coarea-general}.  Equivalently,
one can first take $h$ continuous and compactly supported, apply the
chain rule to an antiderivative of $h$, and extend to nonnegative Borel
functions by a monotone-class argument. Taking $h(r)=r^{-2q-1}$ proves
\eqref{opt:coarea-power}.

Let $U(a)=\max_t(g(t)-at)$.  Its right and left
derivatives are the extreme values of $-t$ over its set of maximizing
times.  Hence at a point where $U$ is differentiable all maximizing times
coincide.  The same argument applies to $\max_t(at-g(t))$.  Both convex
functions are differentiable outside sets of Lebesgue measure zero.
Adding their derivatives proves \eqref{opt:range-derivative} at every
slope outside the union of these exceptional sets.  At exceptional
slopes one may choose the leftmost extremal times; these choices do not
affect any of the integrals above.
\end{proof}

\begin{lemma}\label{opt:extrema-uniqueness}
For every fixed $a\in\R$, the maximum and minimum of $W_t-at$ on
$[0,1]$ are attained at unique, distinct times in $(0,1)$ almost surely.
\end{lemma}

\begin{proof}
For a rational $c\in(0,1)$, the maximum on $[0,c]$
relative to the value at $c$, and the maximum on $[c,1]$ relative to the
same value, are independent suprema of Brownian motions with fixed
drifts, one run backward from $c$.  Each has an atomless distribution.
For zero drift this follows from the reflection principle.  For fixed
drift it follows from equivalence of the path laws: condition on the
endpoint, write the path as a Brownian bridge plus its linear endpoint,
and note that a translated normal endpoint distribution has a strictly
positive density relative to the original one.  Thus these two maxima
cannot be equal with positive probability.  If a path had two distinct
global maximizing times, some rational $c$ would separate them, which
would force just such an equality.  A countable union proves uniqueness.
Apply the argument to $-W$ for the minimum.  Finally, Brownian motion has
both positive and negative values immediately after time zero, by the
reflection principle; time reversal gives the analogous assertion at
time one.  The extrema therefore lie in $(0,1)$ almost surely, and their
times are distinct.  Equivalence of the drifted and undrifted laws
extends the endpoint assertion to every fixed drift.
\end{proof}

\subsection{Gaussian endpoints and ordered extrema}
\label{opt:bridge-subsection}

Write
\[
 W_t=B_t+tZ,\qquad Z=W_1,
\]
where $B$ is a standard Brownian bridge and $Z\sim N(0,1)$ is independent
of $B$.  For a nonconstant path $g$ with specified extremal times, put
\[
 \Delta(g)=|T_+(g)-T_-(g)|,
 \qquad \Phi_q(g)=\frac{\Delta(g)}{\operatorname{range}(g)^{2q+1}}.
\]
The scaling relation in Lemma~\ref{def:scaling} and affine invariance give
$m_q=\E F(W)^{-2q}=\E F(B)^{-2q}$.  Applying
Lemma~\ref{opt:coarea-lemma} to the bridge and using Tonelli's theorem yields
\begin{align}
 m_q
 &=2^{2q}q\int_{\R}\E\Phi_q(B+y\,\cdot)\,dy
 \label{opt:bridge-moment}\\
 &=2^{2q}q\sqrt{2\pi}\,
    \E\left[
      e^{W_1^2/2}\frac{|T_+(W)-T_-(W)|}
      {(\max W-\min W)^{2q+1}}
    \right].
 \label{opt:ordinary-brownian-moment}
\end{align}
The second equality follows by conditioning on $Z$: multiplication by
$\sqrt{2\pi}e^{Z^2/2}$ cancels its density
$(2\pi)^{-1/2}e^{-Z^2/2}$.  All the integrands are nonnegative, so both
identities initially hold in the extended nonnegative reals; their
finiteness follows from Lemma~\ref{def:finiteness}.

On $(0,r)$ let $K_r$ be the killed transition density for generator
$\tfrac12\partial_x^2$:
\begin{equation}\label{opt:heat-kernel}
 K_r(t;x,z)=\frac2r\sum_{n=1}^\infty
  \sin\frac{n\pi x}{r}\sin\frac{n\pi z}{r}
  e^{-n^2\pi^2t/(2r^2)},\qquad t>0.
\end{equation}
The series and all its spatial derivatives converge uniformly when time
is bounded away from zero.  Define inward boundary derivatives by
\begin{align}
 A_0(t,x;r)&=\left.\partial_zK_r(t;x,z)\right|_{z=0},
 \nonumber\\
 A_r(t,x;r)&=-\left.\partial_zK_r(t;x,z)\right|_{z=r},
 \label{opt:boundary-kernels}\\
 C(t;r)&=-\left.\partial_x\partial_zK_r(t;x,z)
                         \right|_{x=r,z=0}.
 \nonumber
\end{align}
In particular,
\begin{align}
 A_0(t,x;r)&=\frac{2\pi}{r^2}\sum_{n=1}^\infty
 n\sin(n\pi x/r)e^{-n^2\pi^2t/(2r^2)},
 \label{opt:boundary-series}\\
 C(t;r)&=\frac{2\pi^2}{r^3}\sum_{n=1}^\infty
 (-1)^{n+1}n^2e^{-n^2\pi^2t/(2r^2)},
 \nonumber
\end{align}
and $A_r(t,x;r)=A_0(t,r-x;r)$.  These kernels are positive, although
their displayed series are not positive term by term.  Positivity of
$A_0,A_r$ follows either from the heat equation and the inward boundary
derivative, or from their interpretation as twice the corresponding
exit-time densities.  Differentiating the semigroup identity gives
\begin{equation}\label{opt:cross-kernel-positive}
 C(t;r)=\int_0^r A_r(t/2,y;r)A_0(t/2,y;r)\,dy>0.
\end{equation}

Translate these definitions to an interval $(l,h)$, writing $A_l,A_h$
for the two inward derivatives and $C_{lh}$ for the derivative across
the two different endpoints.  We retain the endpoint of the Brownian
path in the next density because the weight in
\eqref{opt:ordinary-brownian-moment} depends on that endpoint.

\begin{proposition}\label{opt:ordered-density-proposition}
Let $l<0<h$, $w\in(l,h)$, and $0<u<v<1$.  The part of the joint density
of $(\min W,\max W,T_-(W),T_+(W),W_1)$ for which the minimum is attained
before the maximum is
\begin{equation}\label{opt:ordered-extrema-density}
 \frac14 A_l(u,0)C_{lh}(v-u)A_h(1-v,w)
       \,dl\,dh\,du\,dv\,dw.
\end{equation}
Here $dl$ denotes positive Lebesgue volume for the variable $\min W$.
The other order has the same formula with the two boundary labels
interchanged.  There are no additional time-endpoint or diagonal
components.
\end{proposition}

The proof is given in Appendix~\ref{app:ordered-density}.

The density \eqref{opt:ordered-extrema-density} is consistent with the
five-variable extremum distribution, including the terminal position,
in \cite[Section~2.3, Eq.~(22), and Appendix~D]{SchehrLeDoussal2010},
after converting their variance convention $\E W_t^2=2t$ to ours.
The three-path decomposition is also used in
\cite[Section~III.A]{Mori2020}. Appendix~\ref{app:ordered-density} gives a boundary-variation
derivation with the normalization needed here.
Integrating over the two time orders gives
$-\partial_l\partial_hK_{l,h}(1;0,w)$.  For example, taking a Laplace
transform in total time and putting $x=-l$, $z=w-l$, $r=h-l$,
$k=\sqrt{2\lambda}$, both sides become
\begin{equation}\label{opt:resolvent-normalization}
 \frac{2k}{\sinh^3(kr)}
 \{\sinh(k(r-x))\sinh(kz)
     +\sinh(kx)\sinh(k(r-z))\}.
\end{equation}
This identity follows directly by differentiating the resolvent displayed
in \eqref{opt:resolvent} below, with the endpoints held fixed in the
original coordinates.

\subsection{Positive kernels and Laplace transforms}
\label{opt:positive-integrals-subsection}

Substitute Proposition~\ref{opt:ordered-density-proposition} into
\eqref{opt:ordinary-brownian-moment}, and change variables to
\[
 r=h-l,\qquad x=-l,\qquad z=w-l,\qquad s=v-u.
\]
The spatial and time Jacobians have absolute value one, $0<x,z<r$,
and $u,s>0$, $u+s<1$.  Spatial reflection exchanges the two time orders
and preserves $e^{(z-x)^2/2}$, so they have equal contributions.  Thus
\begin{align}
 m_q={}&2^{2q-1}q\sqrt{2\pi}
  \int_0^\infty\frac{dr}{r^{2q+1}}
  \int_0^r\int_0^r e^{(z-x)^2/2}
 \nonumber\\[-2pt]
 &\qquad\times\int_{\substack{u,s>0\\u+s<1}}
 s A_0(u,x;r)C(s;r)A_r(1-u-s,z;r)
 \,du\,ds\,dx\,dz.
 \label{opt:five-dimensional-moment}
\end{align}
The prefactor is
$2^{2q}q\sqrt{2\pi}\times(1/4)\times2$.
The kernel functions are positive, so Tonelli's theorem applies before
introducing their spectral expansions.

We now use unit-width kernels, suppressing their final argument $r=1$,
and write $A_1$ for the upper boundary kernel.  Define
\begin{equation}\label{opt:convolution-definition}
 Q_{a,b}(T)=\int_{\substack{v,w>0\\v+w<T}}
 A_0(v,a)\,[wC(w)]\,A_1(T-v-w,b)\,dv\,dw,
 \qquad 0<a,b<1.
\end{equation}
The changes of variables
$x=ra$, $z=rb$, $u=r^2v$, $s=r^2w$, $T=r^{-2}$ use the scaling
$A\mapsto r^{-2}A$ and $C\mapsto r^{-3}C$.  The powers of $r$ in
\eqref{opt:five-dimensional-moment}, before changing $dr$, are
\[
 \underbrace{r^{-2q-1}}_{\text{range weight}}
 \underbrace{r^2}_{dx\,dz}
 \underbrace{r^4}_{du\,ds}
 \underbrace{r^2}_{s}
 \underbrace{r^{-7}}_{A_0CA_r}=r^{-2q}.
\]
As $r^{-2q}\,dr=-\tfrac12T^{q-3/2}\,dT$, we obtain the following
formula for every real $q>0$:
\begin{equation}\label{opt:unit-width-moment}
 m_q=2^{2q-2}q\sqrt{2\pi}
 \int_0^1\int_0^1\int_0^\infty
 T^{q-3/2}e^{(b-a)^2/(2T)}Q_{a,b}(T)\,dT\,da\,db.
\end{equation}

The same kernels describe the whole distribution.
\begin{proposition}\label{opt:survival-proposition}
The survival function $S_{\rm opt}(t)=\Pbb\{\tauopt>t\}$ is continuous
for $t>0$ and satisfies
\begin{equation}\label{opt:survival-kernel}
 S_{\rm opt}(4T)=\frac{\sqrt{2\pi}}{4\sqrt T}
   \int_0^1\int_0^1
       e^{(b-a)^2/(2T)}Q_{a,b}(T)\,da\,db,
 \qquad T>0.
\end{equation}
\end{proposition}

\begin{proof}
For a nonnegative compactly supported function $h$,
\eqref{opt:coarea-general} and the endpoint cancellation give
\[
 2\int_0^\infty h(r)\Pbb\{\rho(W)<r\}\,dr
 =\sqrt{2\pi}\,\E[e^{W_1^2/2}\Delta(W)h(\max W-\min W)].
\]
Use the ordered density and the same spatial reflection as above.
After equating the level densities, this says, for almost every $r>0$,
\begin{align*}
 \Pbb\{\rho(W)<r\}
  ={}&\frac{\sqrt{2\pi}}4
   \int_0^r\int_0^r e^{(z-x)^2/2}\\
 &\quad\times\int_{\substack{u,s>0\\u+s<1}}
    s A_0(u,x;r)C(s;r)A_r(1-u-s,z;r)
    \,du\,ds\,dx\,dz.
\end{align*}
The width scaling of the inner integral is $r$.
Setting $r=T^{-1/2}$ therefore gives \eqref{opt:survival-kernel}
almost everywhere, since
$\tauopt\overset{d}=4\rho(W)^{-2}$.

We spell out why this identifies a continuous version and hence holds
at every positive $T$.  On a compact positive $T$ interval, the
exponential factor in \eqref{opt:survival-kernel} is bounded.
The image representation of the heat kernel gives, for $0<t\leq T_0$,
\[
 A_0(t,a)\leq C_0 a t^{-3/2}e^{-a^2/(2t)},\qquad
 A_1(t,b)\leq C_0(1-b)t^{-3/2}e^{-(1-b)^2/(2t)},
\]
and
$C(t)\leq C_1t^{-5/2}e^{-1/(4t)}$ after increasing $C_1$.
The first two bounds also follow by comparison with the corresponding
half-line exit densities.  Their integrals in $a$ and $b$ are bounded
by constants times $t^{-1/2}$; the bound for $C$ follows by
differentiating the image terms, whose spatial displacements across
the endpoints have length at least one.  Rescale $v,w,T-v-w$ by $T$
in \eqref{opt:convolution-definition}.  These bounds give an integrable
majorant on the fixed time simplex and on $(a,b)\in(0,1)^2$, uniform
for $T$ in the chosen compact interval.  Dominated convergence proves
continuity.  A monotone survival function equal almost everywhere to
this continuous function cannot have a jump: its left and right limits
are the same continuous value.  Thus the law is atomless and the
identity holds everywhere.
\end{proof}

For $T>0$ and $d\in\R$, a Gaussian integral gives the positive identity
\begin{equation}\label{opt:gaussian-laplace}
 T^{-1/2}e^{d^2/(2T)}
 =\frac1{\sqrt\pi}\int_0^\infty
   \lambda^{-1/2}\cosh(d\sqrt{2\lambda})e^{-\lambda T}\,d\lambda.
\end{equation}
For example, put $k=\sqrt{2\lambda}$ and evaluate the even Gaussian
integral on the real line. With $k=\sqrt{2\lambda}$, the unit-interval
Dirichlet resolvent is
\begin{equation}\label{opt:resolvent}
 \widehat K_\lambda(x,z)
  =\int_0^\infty e^{-\lambda t}K_1(t;x,z)\,dt
  =\frac{2\sinh(k\min(x,z))\sinh(k(1-\max(x,z)))}
          {k\sinh k}.
\end{equation}
This expression is obtained by solving
$(\lambda-\tfrac12\partial_x^2)\widehat K_\lambda(x,z)=\delta_z(x)$
with zero boundary values; its first derivative has jump $-2$ at $z$.
Taking inward derivatives yields
\begin{equation}\label{opt:boundary-transforms}
 \widehat A_0(\lambda,a)=\frac{2\sinh(k(1-a))}{\sinh k},\qquad
 \widehat A_1(\lambda,b)=\frac{2\sinh(kb)}{\sinh k},\qquad
 \widehat C(\lambda)=\frac{2k}{\sinh k}.
\end{equation}
In particular, with
\begin{equation}\label{opt:f-definition}
 f(k)=\coth k-\frac1k,
\end{equation}
the transform of $tC(t)$ is
$-\partial_\lambda\widehat C=2f(k)/\sinh k$.  Thus, writing the
transform of $Q$ as a function of $k$,
\begin{equation}\label{opt:q-transform}
 \widehat Q_{a,b}(k)
 :=\int_0^\infty e^{-k^2T/2}Q_{a,b}(T)\,dT
 =\frac{8f(k)}{\sinh^3 k}
   \sinh(k(1-a))\sinh(kb).
\end{equation}
All differentiations here are justified for $\lambda>0$: at small
times the cross-boundary kernel has Gaussian decay, and at large
times the spectral kernel has exponential decay.  The same estimates
permit every finite number of $\lambda$ derivatives.

\subsection{The survival function as an image series}

We now reduce \eqref{opt:survival-kernel} to a single series. Write
$E_1(x)=\int_x^\infty e^{-u}\,du/u$ for $x>0$.

\begin{proposition}[An image series]
\label{const:optimal-image-series}
For $m\ge1$, put
\begin{align}
 A_m(t)&=-\frac{8(2m^2+1)}3+\frac t6\left(1-\frac1{m^2}\right),
 \notag\\
 B_m(t)&=\frac{(2m+1)^4}{3m(m+1)}
       +\frac t{12}\left(-2+\frac1{m^2}+\frac1{(m+1)^2}\right),
 \label{const:image-coefficients}\\
 \Psi_m(t)&=A_m(t)e^{-8m^2/t}+B_m(t)e^{-8m(m+1)/t}
 \notag\\[-2pt]
 &\qquad+4m^2E_1(8m^2/t)-4m(m+1)E_1(8m(m+1)/t).
 \notag
\end{align}
Then, for every $t>0$,
\begin{equation}
 S_{\mathrm{opt}}(t)=1+\sum_{m=1}^\infty\Psi_m(t).
 \label{const:optimal-survival-series}
\end{equation}
The series converges absolutely and uniformly on each bounded interval
$0<t\le t_0$, with its continuous extension at zero.
\end{proposition}

\begin{proof}
Set $x=1-a$, $y=b$, $d=x+y-1$, and, for $m\ge1$, let
$b_m=m(m+1)/2$ and $c_m=m(m+1)(2m+1)/6$.
The geometric expansions of $\sinh^{-3}k$ and
$\coth k\,\sinh^{-3}k$ in \eqref{opt:q-transform} give
\[
 \widehat Q_{1-x,y}(k)
 =16\sum_{m\ge1}\sum_{\sigma,\eta\in\{-1,1\}}
 \sigma\eta\left(c_m-\frac{b_m}{k}\right)e^{-Lk},
 \qquad L=2m+1-\sigma x-\eta y.
\]
For $L>0$, the inverse Laplace transforms, with $k=\sqrt{2\lambda}$, are
\[
 e^{-Lk}\longleftrightarrow
 \frac{L e^{-L^2/(2T)}}{\sqrt{2\pi}T^{3/2}},
 \qquad
 \frac{e^{-Lk}}k\longleftrightarrow
 \frac{e^{-L^2/(2T)}}{\sqrt{2\pi T}}.
\]
These follow from the Gaussian integral. Since $L\ge2m-1$, the
inverted series converges locally uniformly in $T>0$ and uniformly
in $(x,y)\in[0,1]^2$. Its Laplace transform may be taken termwise
for sufficiently large $\lambda$, which identifies it with $Q$.
Substitution into \eqref{opt:survival-kernel} yields
\begin{equation}
 S_{\mathrm{opt}}(4T)=\frac4{T^2}\sum_{m\ge1}
 \sum_{\sigma,\eta}\sigma\eta\int_0^1\int_0^1
 (c_mL-b_mT)e^{-(L^2-d^2)/(2T)}\,dx\,dy.
 \label{const:image-spatial-integral}
\end{equation}

For completeness, the spatial integrations can be performed as follows.
Put $h=2/T$. For the two equal signs, $L^2-d^2$ depends only on
$x+y$. Combining the $++$ term at $m$ with the $--$ term at $m-1$
gives
\[
 C_m\{2mF_m-D_m\}-m^2TF_m,
 \qquad C_m=\frac{m(2m^2+1)}3,
\]
where the absent term at $m=1$ is zero, and
\[
 F_m=e^{-hm^2}\frac{2(\cosh(hm)-1)}{h^2m^2},\qquad
 D_m=e^{-hm^2}
 \left.\frac{d}{dz}\frac{2(\cosh z-1)}{z^2}\right|_{z=hm}.
\]
Here $F_m$ and $D_m$ are the integrals with weights $1-|d|$ and
$d(1-|d|)$ over $-1<d<1$. Each mixed-sign term is
\begin{align*}
 J_m&=\int_m^{m+1}\int_m^{m+1}
       \{c_m(u+v)-b_mT\}e^{-huv}\,du\,dv\\
 &=\frac{c_mT^2}{2}\left\{
 \frac{e^{-hm^2}}m-\left(\frac1m+\frac1{m+1}\right)e^{-hm(m+1)}
 +\frac{e^{-h(m+1)^2}}{m+1}\right\}\\
 &\quad-\frac{b_mT^2}{2}
 \{E_1(hm^2)-2E_1(hm(m+1))+E_1(h(m+1)^2)\}.
\end{align*}
This follows by first integrating $e^{-huv}$ in $v$.
Thus \eqref{const:image-spatial-integral} is
$4T^{-2}\sum_m\{C_m(2mF_m-D_m)-m^2TF_m-2J_m\}$.
Collecting the square and consecutive-product exponents, with
$t=4T$, gives \eqref{const:image-coefficients} and
\eqref{const:optimal-survival-series}. The constant $1$ comes from
the exponent $m(m-1)=0$ at $m=1$.
Finally, $E_1(x)\le e^{-x}/x$ bounds the absolute values of the
summands by $C_{t_0}(1+m^2)e^{-8m^2/t_0}$ on $0<t\le t_0$.
This proves the asserted convergence and justifies the rearrangements.
\end{proof}

\subsection{The first four moments}
\label{opt:raw-moment-evaluation}

For a positive integer $q$, apply \eqref{opt:gaussian-laplace} to
\eqref{opt:unit-width-moment}.  The transform of
$T^{q-1}Q(T)$ is $(-\partial_\lambda)^{q-1}\widehat Q$, and all
integrands before this differentiation are nonnegative.  Tonelli's
theorem and $\lambda^{-1/2}d\lambda=\sqrt2\,dk$ therefore give
\begin{equation}\label{opt:integer-moment-derivatives}
 m_q=2^{2q-1}q\int_0^\infty\int_0^1\int_0^1
 \cosh(k(b-a))
 \left(-\frac1k\frac{d}{dk}\right)^{q-1}
 \widehat Q_{a,b}(k)\,da\,db\,dk.
\end{equation}
The derivatives in this display act only on $\widehat Q$, not on the
factor $\cosh(k(b-a))$.

We evaluate all four moments by the same spatial integration and contour
argument.  Introduce a separate parameter for the hyperbolic cosine
weight:
\begin{align}
 J(k,h)
 &:=\int_0^1\int_0^1
     \cosh(h(b-a))\sinh(k(1-a))\sinh(kb)\,da\,db
 \notag\\
 &=\frac{\bigl[(k^2+h^2)\sinh^2k+2k^2\bigr]\cosh h
       -2k^2\cosh k-2kh\sinh k\cosh k\sinh h}
       {(k^2-h^2)^2}.
 \label{opt:spatial-generating-function}
\end{align}
The second line follows by expanding $\cosh(h(b-a))$ and integrating
the resulting products of single-variable hyperbolic functions.
The integral in the first line shows that the singularities at
$h=\pm k$ in the second line are removable.  Set
\begin{equation}\label{opt:spatial-recursion}
 B_0(k,h)=\frac{8f(k)}{\sinh^3k}J(k,h),\qquad
 B_{j+1}(k,h)=-\frac1k\partial_k B_j(k,h),\qquad
 L_q(k)=B_{q-1}(k,k).
\end{equation}
The parameter $h$ is held fixed throughout this recursion and is set
equal to $k$ only after all derivatives have been taken.  Hence
\eqref{opt:integer-moment-derivatives} becomes
\begin{equation}\label{opt:one-dimensional-moments}
 m_q=2^{2q-1}q\int_0^\infty L_q(k)\,dk.
\end{equation}

Evaluation of \eqref{opt:spatial-recursion} gives the following
four integrands. In these formulas
$f=f(k)$ and all summands are kept together at zero:
\begin{align}
 L_1={}&2f^4-2f^2+\frac{2f^3+2f}{k},
 \label{opt:raw-integrands}\\
 L_2={}&\frac{8f^5-12f^3+4f}{k}
       +\frac{20f^4-10f^2-2}{k^2}
       +\frac{12f^3+6f}{k^3},
 \notag\\
 L_3={}&\frac{40f^6-227f^4/3+119f^2/3-4}{k^2}
       +\frac{160f^5-517f^3/3+29f}{k^3}
 \notag\\
 &\quad+\frac{215f^4-230f^2/3-10}{k^4}
       +\frac{95f^3+30f}{k^5},
 \notag\\
 L_4={}&\frac{240f^7-540f^5+378f^3-78f}{k^3}
       +\frac{1320f^6-1995f^4+746f^2-35}{k^4}
 \notag\\
 &\quad+\frac{2772f^5-2379f^3+267f}{k^5}
       +\frac{2628f^4-756f^2-70}{k^6}
       +\frac{936f^3+210f}{k^7}.
 \notag
\end{align}
These identities require only differentiation of
\eqref{opt:spatial-generating-function} and the relation
\begin{equation}\label{opt:f-ode-asymptotics}
 f'=1-f^2-\frac{2f}{k},\qquad
 f(k)=\frac{k}{3}-\frac{k^3}{45}+O(k^5).
\end{equation}

\begin{lemma}[Residue evaluation]\label{opt:raw-residue-lemma}
The functions $L_1,\ldots,L_4$ are even and meromorphic, with removable
singularities at zero and poles only at $i\pi n$, $n\in\mathbb Z\setminus\{0\}$.
Writing $z_n=i\pi n$, their values and residues are
\begin{equation}\label{opt:residue-table}
\begin{array}{c|c|l}
 q&L_q(0)&\operatorname{Res}_{z=z_n}L_q(z)\\ \hline
 1&2/3&-20z_n^{-3}\\
 2&8/15&\dfrac{14}{3}z_n^{-3}+140z_n^{-5}\\
 3&167/315&-\dfrac29z_n^{-3}-\dfrac{170}{3}z_n^{-5}-1260z_n^{-7}\\
 4&598/945&6z_n^{-5}+714z_n^{-7}+13860z_n^{-9}
\end{array}
\end{equation}
Moreover,
\begin{equation}\label{opt:raw-residue-evaluation}
 \int_0^\infty L_q(k)\,dk
 =\pi i\sum_{n=1}^\infty\operatorname{Res}_{z=z_n}L_q(z)
   -\frac23\,\mathbf 1_{\{q=1\}}.
\end{equation}
\end{lemma}

\begin{proof}
Since $f$ is odd, every displayed $L_q$ is even.  Removability at zero
also follows directly from the integral defining $J$: the function
$B_0(k,h)$ is analytic in $(k^2,h^2)$ near $(0,0)$, and
$-k^{-1}\partial_k=-2\partial_{k^2}$ preserves that property.
Taylor expansion gives the second column of
\eqref{opt:residue-table}.  At a nonzero pole put $z=z_n+w$ and use
\begin{equation}\label{opt:pole-expansion}
 f(z_n+w)=\frac1w+\frac w3-\frac{w^3}{45}
           +\frac{2w^5}{945}-\frac1{z_n+w}+O(w^7).
\end{equation}
Substituting this finite expansion into \eqref{opt:raw-integrands} and
collecting the coefficient of $w^{-1}$ gives the third column.  For
example, substitution into
$L_1=2f^4-2f^2+2(f^3+f)/z$ cancels the $z_n^{-1}$ terms and leaves
$-20z_n^{-3}$.  The expansion through $w^5$ suffices for $L_4$, whose
highest power of $f$ is seven.  In particular, every residue series
in \eqref{opt:raw-residue-evaluation} is absolutely convergent.

To justify the contour and its boundary term, integrate $L_q$ around
the rectangle with vertices $-M,M,M+iR_N,-M+iR_N$, where
$R_N=\pi(N+1/2)$.  For fixed $N$, the two vertical integrals tend to
zero as $M\to\infty$.  Indeed, uniformly on those sides,
$\coth z=\operatorname{sgn}(\Re z)+O(e^{-2M})$;
substitution into \eqref{opt:raw-integrands} gives decay at least
$O(M^{-1})$ for $q\ge2$ and $O(M^{-2})+O(e^{-2M})$ for $q=1$.
The residue theorem \cite[\S1.10(iv)]{DLMF} therefore gives
\begin{equation}\label{opt:rectangle-identity}
 2\int_0^\infty L_q(x)\,dx
 -\int_{\R}L_q(x+iR_N)\,dx
 =2\pi i\sum_{n=1}^{N}\operatorname{Res}_{z=z_n}L_q(z).
\end{equation}

On the top edge, $\coth(x+iR_N)=\tanh x$.  Put $z=x+iR_N$ and
$u=\tanh x$, so $f(z)=u-z^{-1}$.  For $q=1$, expansion gives the exact
identity
\[
 L_1(z)=2u^4-2u^2
       +\frac{6u(1-u^2)}{z}
       +\frac{6u^2-4}{z^2}-\frac{2u}{z^3}.
\]
The integral of its first term is
\[
 \int_{\R}(2\tanh^4x-2\tanh^2x)\,dx
   =-2\int_{-1}^1u^2\,du=-\frac43.
\]
The remaining integrals are $O(R_N^{-1})$: the numerator of the
$z^{-1}$ term is integrable, and
$\int_{\R}|x+iR_N|^{-p}\,dx=O(R_N^{1-p})$ for $p>1$.
For $q=2$ the coefficient of $z^{-1}$ is
\[
 8u^5-12u^3+4u=4u(u^2-1)(2u^2-1),
\]
which is integrable in $x$; all remaining coefficients are bounded
and multiply $z^{-p}$ with $p\ge2$.  Its top integral therefore tends
to zero.  The formulas for $q=3,4$ have only powers $z^{-p}$ with
$p\ge2$, again with bounded coefficients after substituting
$f=u-z^{-1}$, and give the same conclusion.
Letting $N\to\infty$ in \eqref{opt:rectangle-identity} proves
\eqref{opt:raw-residue-evaluation}.
\end{proof}

Finally, for $j\ge1$,
\[
 \pi i\sum_{n=1}^\infty(i\pi n)^{-(2j+1)}
     =(-1)^j\frac{\zeta(2j+1)}{\pi^{2j}}.
\]
Multiplication of \eqref{opt:raw-residue-evaluation} by
$2^{2q-1}q$, followed by substitution of \eqref{opt:residue-table},
gives the four formulas in \eqref{const:optimal-moments}.
In particular, the rational term $-4/3$ in $m_1$ is exactly the
nonvanishing top-edge contribution; the other three moments have no
such boundary term.

\section{The anchored functional}
\label{an:section}

We calculate the distribution of $D_{\mathrm{an}}(W,1)$ and all positive
integer moments of $\tauan$ from Definition~\ref{def:lifetimes}.

\subsection{Time inversion and an expanding strip}

\begin{lemma}[A sum of dependent drifted suprema]
\label{an:time-inversion}
For a standard Brownian motion $X$, define
\begin{equation}
 A=\sup_{u\ge0}(X_u-u),\qquad
 B=\sup_{u\ge0}(-X_u-u),\qquad S=A+B.
 \label{an:suprema}
\end{equation}
Then
\begin{equation}
 D_{\mathrm{an}}(W,1)^2\ \stackrel{d}{=}\ \frac S2,
 \qquad \tauan\ \stackrel{d}{=}\ \frac2S.
 \label{an:inverse-sum}
\end{equation}
Both $A$ and $B$ have exponential distributions of rate $2$.
\end{lemma}

\begin{proof}
The Gaussian process $\widetilde W_s=sW_{1/s}$, $s>0$, is standard
Brownian motion: its covariance is $\min(s,t)$, and its continuous
extension at zero follows from the Brownian strong law. For $e>0$,
the event $D_{\mathrm{an}}(W,1)\le e$ is equivalent to
\[
 \bigcap_{s\ge1}
 [\widetilde W_s-es,\widetilde W_s+es]\ne\varnothing.
\]
With $X_u=\widetilde W_{1+u}-\widetilde W_1$, this condition becomes
\[
 \sup_{u\ge0}(X_u-eu)+\sup_{u\ge0}(-X_u-eu)\le2e.
\]
Brownian scaling applies jointly to these two suprema, so their sum has
the distribution $S/e$. This proves the first identity in
\eqref{an:inverse-sum}. The one-sided hitting probability
$\Pbb(\sup_{u\ge0}(X_u-u)>a)=e^{-2a}$ follows by stopping the
exponential martingale $\exp(2X_u-2u)$ at two finite barriers and then
sending the lower barrier to $-\infty$. Thus $A,B$ are positive and
finite almost surely.

The second identity in \eqref{an:inverse-sum} follows from
Lemma~\ref{def:scaling}.
\end{proof}

\begin{proposition}[The equal-slope wedge probability]
\label{an:wedge}
The equal-slope case of the classical wedge formula
\cite[Section~2]{YcartDrouilhet2016} is, for $a,b>0$,
\begin{equation}
 G(a,b):=\Pbb(A\le a,B\le b)
 =1+2\sum_{n=1}^{\infty}(-1)^n
 e^{-n^2(a+b)}\cosh\bigl(n(a-b)\bigr).
 \label{an:wedge-series}
\end{equation}
\end{proposition}

\begin{proof}
The event in question is
$-b-t\le X_t\le a+t$ for every $t\ge0$. Write $s=a+b$, set
$T=1/(2s)$, and make the deterministic transformation
\begin{equation}
 \rho(t)=\frac{t}{s(s+2t)},\qquad
 Z_{\rho(t)}=\frac{X_t+b+t}{s+2t}.
 \label{an:bridge-transform}
\end{equation}
Its mean is the linear interpolation from $b/s$ to $1/2$ over $[0,T]$.
For $t\le u$, its covariance is
\[
 \frac{t}{(s+2t)(s+2u)}
 =\rho(t)\left(1-\frac{\rho(u)}T\right).
\]
Consequently $Z$ is a standard Brownian bridge of duration $T$, from
$b/s$ to $1/2$. It extends continuously to that endpoint because
$X_t/t\to0$ almost surely. The wedge condition becomes the condition
that the bridge remains in $[0,1]$.

Let $p_T(x,y)=(2\pi T)^{-1/2}\exp(-(y-x)^2/(2T))$, and let $K_1$
be the Dirichlet heat kernel on $(0,1)$ for the generator
$\tfrac12\partial_{xx}$. Conditioning Brownian motion on its endpoint
gives
\[
 G(a,b)=\frac{K_1(T;b/s,1/2)}{p_T(b/s,1/2)}.
\]
The reflection formula
\[
 K_1(T;x,y)=\sum_{j\in\mathbb Z}
 \{p_T(x,y+2j)-p_T(x,-y+2j)\}
\]
is absolutely convergent at fixed $T>0$. Dividing by the free kernel
and separating the even and odd images yields
\[
 G(a,b)=\sum_{n\in\mathbb Z}(-1)^n
 \exp\{-n^2s+n(a-b)\},
\]
which is \eqref{an:wedge-series}. The distinction between staying in
the closed and open strip has probability zero, since the initial and
terminal points are interior and a Brownian bridge cannot touch a
boundary without crossing it with positive probability.
\end{proof}

\subsection{The distribution and its Laplace transform}

\begin{theorem}[An explicit survival function]
\label{an:distribution}
For $s,t>0$,
\begin{align}
 \Pbb(S\le s)
 &=\sum_{n=0}^{\infty}(-1)^n(2n+1)e^{-n(n+1)s}
   =\prod_{n=1}^{\infty}(1-e^{-2ns})^3,
 \label{an:sum-cdf}\\
 \Pbb(\tauan>t)
 &=\prod_{n=1}^{\infty}(1-e^{-4n/t})^3.
 \label{an:survival}
\end{align}
In particular, $S$ and $\tauan$ have no atoms on $(0,\infty)$.
\end{theorem}

\begin{proof}
Use coordinates $s=a+b$ and $d=a-b$, so that
$\partial_a\partial_b=\partial_s^2-\partial_d^2$.
For each compact subinterval of $s>0$, the series
\eqref{an:wedge-series} and the derivatives needed here converge
absolutely and uniformly for $|d|\le s$.
The marginal exponential laws imply that no probability is carried by
the coordinate axes. Integrating the joint density along $a+b=s$
therefore gives
\begin{align*}
 f_S(s)
 &=\int_0^s\partial_a\partial_bG(a,s-a)\,da\\
 &=\sum_{n\ge1}(-1)^n n(n^2-1)
   \{e^{-n(n-1)s}-e^{-n(n+1)s}\}\\
 &=\sum_{n\ge1}(-1)^{n+1}n(n+1)(2n+1)e^{-n(n+1)s}.
\end{align*}
Integration from $s$ to infinity gives the series in
\eqref{an:sum-cdf}. The identity between this series and the product
follows by differentiating at zero the series and product formulas
for the Jacobi theta function $\theta_1$, and cancelling their common
factor; see \cite[\S\S20.2 and 20.5]{DLMF}. Equivalently, it is
Jacobi's identity
\[
 \sum_{n\ge0}(-1)^n(2n+1)q^{n(n+1)/2}
 =\prod_{n\ge1}(1-q^n)^3,
 \qquad 0<q<1,
\]
with $q=e^{-2s}$. The final assertion and
\eqref{an:survival} now follow from \eqref{an:inverse-sum}.
\end{proof}

\begin{proposition}[Laplace transform]
\label{an:laplace-proposition}
For $z>0$,
\begin{equation}
 L_S(z):=\E[e^{-zS}]
 =\frac{\pi z}{\cos\bigl(\frac\pi2\sqrt{1-4z}\bigr)}.
 \label{an:laplace}
\end{equation}
For $z>1/4$ the denominator is understood as
$\cosh\bigl(\frac\pi2\sqrt{4z-1}\bigr)$; the expression extends
continuously to $L_S(0)=1$.
\end{proposition}

\begin{proof}
Integration by parts gives
$L_S(z)=z\int_0^\infty e^{-zs}\Pbb(S\le s)\,ds$.
To justify use of the series \eqref{an:sum-cdf}, first restrict the
integral to $[\delta,\infty)$, where $\delta>0$. Absolute convergence
there gives
\[
 z\sum_{n\ge0}\frac{(-1)^n(2n+1)
 e^{-[z+n(n+1)]\delta}}{z+n(n+1)}.
\]
The positive sequence $(2n+1)/(z+n(n+1))$ is eventually decreasing
to zero. After multiplication by the exponential factor it remains
eventually decreasing, and its alternating-series remainder is
bounded uniformly for $\delta\ge0$ by the first omitted coefficient
at $\delta=0$. Passing to $\delta\downarrow0$ is therefore legitimate,
and gives the conditionally convergent identity
\begin{equation}
 L_S(z)=z\sum_{n\ge0}
 \frac{(-1)^n(2n+1)}{z+n(n+1)}.
 \label{an:conditional-series}
\end{equation}
Put $\rho=(1-\sqrt{1-4z})/2$. Then $\rho(1-\rho)=z$, and
\[
 \frac{2n+1}{z+n(n+1)}
 =\frac1{n+\rho}+\frac1{n+1-\rho}.
\]
The paired partial-fraction expansion of the cosecant
\cite[Eq.~4.22.5]{DLMF} gives
\[
 \sum_{n\ge0}(-1)^n
 \left(\frac1{n+\rho}+\frac1{n+1-\rho}\right)
 =\frac\pi{\sin(\pi\rho)}.
\]
Since $\sin(\pi\rho)=\cos(\tfrac\pi2\sqrt{1-4z})$, this proves
\eqref{an:laplace}.
\end{proof}

\begin{corollary}[An exponential-series representation]
\label{an:exponential-series}
If $E_n$, $n\ge1$, are independent exponential random variables with
rates $n(n+1)$, then
\begin{equation}
 S\stackrel d=\sum_{n=1}^{\infty}E_n.
 \label{an:exponential-sum}
\end{equation}
\end{corollary}

\begin{proof}
The sum is finite almost surely because
$\sum_{n\ge1}\E E_n=\sum_{n\ge1}1/(n(n+1))=1$.
For $z>0$, put $\rho=(1-\sqrt{1-4z})/2$. The Euler product,
or the gamma reflection formula applied to the
factors $(n+\rho)(n+1-\rho)/(n(n+1))$, gives
\[
 \prod_{n\ge1}\left(1+\frac z{n(n+1)}\right)
 =\frac{\sin(\pi\rho)}{\pi z}.
\]
The Laplace transform of the sum is the reciprocal product, which
agrees with \eqref{an:laplace}. Uniqueness of Laplace transforms
proves the assertion.
\end{proof}

The absorption time $T_{\mathrm K}$ of Kingman's coalescent started
with infinitely many lineages and pairwise merger rate one is a sum
of independent exponentials of rates $k(k-1)/2$, $k\ge2$
\cite[Section~2]{MoehlePitters2015}. Thus
\begin{equation}
 S\stackrel d=\frac{T_{\mathrm K}}2,
 \qquad \tauan\stackrel d=\frac4{T_{\mathrm K}}.
 \label{an:kingman}
\end{equation}
Thus the series in \eqref{an:sum-cdf} is the classical Kingman
distribution in a rescaled variable. Positive integer moments of
$T_{\mathrm K}$ are evaluated in \cite[Theorem~2.2]{MoehlePitters2015};
single-integral representations are given in
\cite[Theorem~2.1]{PoganyNadarajah2017}.
The positive moments of $\tauan$ calculated below are instead the
scaled inverse moments $4^q\E[T_{\mathrm K}^{-q}]$, $q\ge1$.

\subsection{Integer moments}

\begin{theorem}[Moment formula]
\label{an:moments}
For a positive integer $q$, define the even polynomial
\begin{align}
 P_q(y)
 &=\operatorname{Im}\{(y+i)(y^2+2iy)^q\}\notag\\
 &=\sum_{\ell=0}^{\lfloor q/2\rfloor}
 (-1)^\ell4^\ell
 \left\{\binom q{2\ell}+2\binom q{2\ell+1}\right\}
 y^{2(q-\ell)},
 \label{an:moment-polynomial}
\end{align}
where a binomial coefficient outside its usual range is zero. Then
\begin{equation}
 m_{q,\mathrm{an}}
 =\frac\pi{2^{q+1}\Gamma(q)}
 \int_0^\infty\frac{P_q(y)}{\sinh(\pi y/2)}\,dy.
 \label{an:moment-integral}
\end{equation}
In particular, every positive integer moment is a finite rational
linear combination of $\zeta(2j+1)/\pi^{2j}$.
\end{theorem}

\begin{proof}
For every real $q>0$, the positive gamma integral and Tonelli's theorem
give
\begin{equation}
 m_{q,\mathrm{an}}
 =\frac{2^q}{\Gamma(q)}\int_0^\infty z^{q-1}L_S(z)\,dz.
 \label{an:gamma-integral}
\end{equation}
Formula \eqref{an:laplace} shows that this integral converges both at
zero and at infinity. Suppose now that $q$ is an integer and put
$x=\sqrt{4z-1}$. As $z$ runs from zero to infinity, $x$ follows the
path from $i$ down to zero and then along the positive real axis.
The constants in \eqref{an:gamma-integral} become
\begin{equation}
 m_{q,\mathrm{an}}
 =\frac\pi{2^{q+1}\Gamma(q)}
 \int_{i\to0\to\infty}
 \frac{x(x^2+1)^q}{\cosh(\pi x/2)}\,dx.
 \label{an:contour}
\end{equation}
The only zero of the denominator in the closed strip
$0\le\operatorname{Im}x\le1$ is $x=i$. It is cancelled by the
factor $(x^2+1)^q$. The integrand is therefore analytic there after
removing that singularity. On the vertical segment with real part $R$,
the integrand is bounded in absolute value by
$C_q(1+R)^{2q+1}e^{-\pi R/2}$, uniformly in its imaginary part.
Its integral therefore tends to zero as $R\to\infty$.
Cauchy's theorem consequently replaces the path in
\eqref{an:contour} by $x=y+i$, $y\ge0$.

Using $\cosh(\pi(y+i)/2)=i\sinh(\pi y/2)$ and taking real parts
gives \eqref{an:moment-integral}; expansion gives
\eqref{an:moment-polynomial}. Every monomial in $P_q$ has positive
even degree, so its integral converges separately at zero and at
infinity. For $j\ge1$, the positive geometric expansion of the
reciprocal hyperbolic sine gives
\begin{equation}
 \int_0^\infty\frac{y^{2j}}{\sinh(\pi y/2)}\,dy
 =\frac{2(2j)!(2^{2j+1}-1)}{\pi^{2j+1}}\zeta(2j+1).
 \label{an:zeta-integral}
\end{equation}
Substitution proves the final assertion.
\end{proof}

For the first four moments, \eqref{an:moment-polynomial} gives
\begin{align*}
 P_1(y)&=3y^2,& P_2(y)&=5y^4-4y^2,\\
 P_3(y)&=7y^6-20y^4,& P_4(y)&=9y^8-56y^6+16y^4.
\end{align*}
Substitution into \eqref{an:moment-integral} and
\eqref{an:zeta-integral} gives the evaluations in
\eqref{const:anchored-moments}.

\section{Absolute central moments}
\label{abs:section}

Let $\gamma_A$ be the skewness defined in
Lemma~\ref{def:absolute-third}. That lemma reduces the absolute third
central moment to an integral truncated at the mean. We evaluate both
models using three integration kernels and the exponential integral $E_1$.

\begin{lemma}[Three integration kernels]
\label{abs:elementary-kernels}
For $x>0$, define
\begin{align}
 R_0(x)&=\frac{(x^2+5x+2)e^{-x}-x(x^2+6x+6)E_1(x)}6,
 \notag\\
 R_1(x)&=\frac{(2-6x-7x^2-x^3)e^{-x}
                   +x^2(x+2)(x+6)E_1(x)}{24},
 \label{abs:kernels}\\
 R_E(x)&=\frac{(x^3+9x^2+18x+6)E_1(x)
                   -(x^2+8x+11)e^{-x}}{18}.
 \notag
\end{align}
Then
\begin{align*}
 R_j(x)&=\int_0^1u^j(1-u)^2e^{-x/u}\,du
                      \qquad(j=0,1),\\
 R_E(x)&=\int_0^1(1-u)^2E_1(x/u)\,du.
\end{align*}
In particular, these three kernels are positive, and
$R_0(x)\le e^{-x}/3$.
\end{lemma}

\begin{proof}
Write $H_j(x)=\int_0^1u^j e^{-x/u}\,du$, for $j\ge-1$.
The substitution $v=x/u$ gives $H_{-1}(x)=E_1(x)$, and
integration of the derivative of $u^{j+1}e^{-x/u}$ gives
\[
 H_j(x)=\frac{e^{-x}-xH_{j-1}(x)}{j+1}\qquad(j\ge0).
\]
Thus $R_j=H_j-2H_{j+1}+H_{j+2}$ for $j=0,1$.
Since $dE_1(x/u)/du=e^{-x/u}/u$, integration by parts gives
\[
 R_E(x)=\tfrac13\{H_{-1}(x)-3H_0(x)+3H_1(x)-H_2(x)\}.
\]
The recursion yields \eqref{abs:kernels}. The bound on $R_0$
follows from $e^{-x/u}\le e^{-x}$ on $0<u\le1$.
\end{proof}

\begin{theorem}[Exact series for the absolute third moments]
\label{abs:series}
For $A\in\{\mathrm{opt},\mathrm{an}\}$,
\begin{equation}
 \beta_A=\gamma_A-\frac{6\kappa_A^3}{V_A^{3/2}}
                  \sum_{m=1}^\infty\mathcal I_{m,A},
 \label{abs:beta-series}
\end{equation}
where the series converges absolutely.

For the unrestricted case, write the coefficients in
\eqref{const:image-coefficients} as
$A_m(t)=A_m^{(0)}+tA_m^{(1)}$ and
$B_m(t)=B_m^{(0)}+tB_m^{(1)}$, and put
$x_m=8m^2/\kopt$, $y_m=8m(m+1)/\kopt$. Then
\begin{align}
 \mathcal I_{m,\mathrm{opt}}
 &=A_m^{(0)}R_0(x_m)+\kopt A_m^{(1)}R_1(x_m)
     +B_m^{(0)}R_0(y_m)+\kopt B_m^{(1)}R_1(y_m)
 \notag\\
 &\quad+4m^2R_E(x_m)-4m(m+1)R_E(y_m).
 \label{abs:opt-terms}
\end{align}
For the anchored case,
\begin{equation}
 \mathcal I_{m,\mathrm{an}}
 =(-1)^m(2m+1)R_0\!\left(\frac{2m(m+1)}{\kan}\right).
 \label{abs:an-terms}
\end{equation}
\end{theorem}

\begin{proof}
Lemma~\ref{def:absolute-third} can be written as
\[
 \beta_A=\gamma_A+\frac6{V_A^{3/2}}
                \int_0^{\kappa_A}(\kappa_A-t)^2\{1-S_A(t)\}\,dt.
\]
For the unrestricted case, insert
\eqref{const:optimal-survival-series}, set $t=\kopt u$, and use
Lemma~\ref{abs:elementary-kernels}. For the anchored case,
\eqref{an:sum-cdf} and $\tauan\stackrel d=2/S$ give
\[
 S_{\mathrm{an}}(t)
 =1+\sum_{m\ge1}(-1)^m(2m+1)e^{-2m(m+1)/t}.
\]
This series converges absolutely and uniformly on
$0<t\le\kan$, with its continuous extension at zero.
The same substitution therefore gives \eqref{abs:an-terms}.
The uniform absolute convergence in both cases justifies
termwise integration.
\end{proof}

\begin{proposition}[Explicit truncation bounds]
\label{abs:truncation}
Let $\beta_{A,N}$ denote \eqref{abs:beta-series} truncated at $m=N$.
For the unrestricted case, put
$\kappa=\kopt$, $V=V_{\mathrm{opt}}$, $M=N+1$,
$r=e^{-8/\kappa}$ and $p=r^{2M+1}$. Then
\begin{equation}
 |\beta_{\mathrm{opt}}-\beta_{\mathrm{opt},N}|
 \le\frac{54\kappa^3}{V^{3/2}}r^{M^2}
 \left\{\frac{M^2}{1-p}+\frac{2Mp}{(1-p)^2}
                     +\frac{p(1+p)}{(1-p)^3}\right\}.
 \label{abs:opt-error}
\end{equation}
For the anchored case, put $\kappa=\kan$, $V=V_{\mathrm{an}}$,
and $x_N=2(N+1)(N+2)/\kappa$. For $N\ge1$,
\begin{equation}
 0\le(-1)^N(\beta_{\mathrm{an}}-\beta_{\mathrm{an},N})
 \le\frac{6\kappa^3}{V^{3/2}}(2N+3)R_0(x_N)
 \le\frac{2\kappa^3}{V^{3/2}}(2N+3)e^{-x_N}.
 \label{abs:an-error}
\end{equation}
In particular, $N=5$ in \eqref{abs:opt-error} gives a bound below
$1.2\times10^{-31}$, while $N=9$ in \eqref{abs:an-error} gives a bound
below $1.5\times10^{-35}$.
\end{proposition}

\begin{proof}
For $0<t\le\kopt<4$, the coefficients in
\eqref{const:image-coefficients} and $E_1(x)\le e^{-x}/x$ give
\[
 |\Psi_m(t)|\le
 \left(\frac{43}{2}m^2+\frac{4\kopt}{3}\right)e^{-8m^2/\kopt}
 \le27m^2e^{-8m^2/\kopt}.
\]
Therefore $|\mathcal I_{m,\mathrm{opt}}|\le9m^2r^{m^2}$.
For $m=M+j$, use $m^2\ge M^2+(2M+1)j$ and sum
$(M+j)^2p^j$. This proves \eqref{abs:opt-error}.

For the anchored case, set $c_m=(2m+1)R_0(2m(m+1)/\kan)>0$.
The integral for $R_0$ gives
\[
 \frac{c_{m+1}}{c_m}
 \le\frac{2m+3}{2m+1}e^{-4(m+1)/\kan}<1
 \qquad(m\ge1),
\]
because $\kan<3$. Hence the terms alternate with decreasing
absolute values. The alternating-series remainder estimate and
$R_0(x)\le e^{-x}/3$ give \eqref{abs:an-error}.
\end{proof}

Evaluation at 90-digit working precision gives
\[
 \beta_{\mathrm{opt}}\approx1.86013967989860,
 \qquad
 \beta_{\mathrm{an}}\approx1.94285545926864.
\]
Each summand is a finite combination of elementary functions and
$E_1$; the full expressions remain infinite series. No finite
closed form for either $\beta_A$ is established here.
The displayed error bounds control series truncation, not
floating-point rounding.

\paragraph{Use of computational tools.}
OpenAI Codex was used to assist with mathematical derivations, symbolic and
numerical checks, and preparation of the manuscript.
\appendix
\section{The ordered-extremum density}
\label{app:ordered-density}

\begin{proof}[Proof of Proposition~\ref{opt:ordered-density-proposition}]
We derive the density by two boundary variations and then identify the
time variables.  This also fixes the factor $1/4$ without conditioning
directly on an absorbing boundary.

Choose smooth functions $a,b$ supported in a compact subset of $(0,1)$,
and replace the interval at time $t$ by
\[
 (l-\eta b(t),\ h+\epsilon a(t)).
\]
For sufficiently small real $\epsilon,\eta$, this is a nondegenerate
interval containing the fixed initial and terminal positions at times
zero and one.  Denote the resulting killed transition density, started
at zero, by $u^{\epsilon,\eta}(t,z)$.  Its parameter derivatives at zero
may be obtained by differentiating the heat equation in the interior
and its Dirichlet boundary data.  Here is a justification in this
particular setting.  Map the moving interval to $(0,1)$ by
\[
 z=l-\eta b(t)+L(t)y,\qquad
 L(t)=h-l+\epsilon a(t)+\eta b(t).
\]
The transformed equation has second-order coefficient $1/(2L(t)^2)$
and first-order coefficient
$[-\eta b'(t)+L'(t)y]/L(t)$.  They are smooth in the parameters,
uniformly parabolic near zero, and equal the unperturbed coefficients
near the initial time.  Starting at a positive time before the support
of $a,b$, the initial kernel is smooth.  The Duhamel formula on the fixed
interval, or its successive difference quotients, then yields first
and mixed second parameter derivatives, with the differentiated
equations and boundary conditions below.  The compact time support
avoids any differentiation of the point initial data.

For the fixed interval, the solution with zero initial data and
boundary values $q_l(t),q_h(t)$ is
\begin{equation}\label{opt:heat-poisson}
 \frac12\int_0^t
 \{q_l(s)A_l(t-s,z)+q_h(s)A_h(t-s,z)\}\,ds.
\end{equation}
This is Green's formula for the heat equation: the spatial integration
by parts contributes $1/2$, the coefficient of $\partial_z^2$ in the
generator.  Differentiating the upper boundary condition gives, for
$U=\partial_\epsilon u^{\epsilon,0}|_0$,
\[
 U(t,h)=-a(t)\partial_zK_{l,h}(t;0,h)
        =a(t)A_h(t,0),\qquad U(t,l)=0.
\]
The analogous lower derivative $V=\partial_\eta u^{0,\eta}|_0$ has
$V(t,l)=b(t)A_l(t,0)$ and $V(t,h)=0$.  In particular,
\begin{equation}\label{opt:first-shape-derivative}
 U(1,w)=\frac12\int_0^1
  a(t)A_h(t,0)A_h(1-t,w)\,dt.
\end{equation}

Let $H=\partial_\epsilon\partial_\eta
u^{\epsilon,\eta}|_0$.  The upper boundary value of $H$ is
$-a(t)\partial_zV(t,h)$, and the lower boundary value is
$b(t)\partial_zU(t,l)$.  Substitute \eqref{opt:heat-poisson} for $U,V$
and differentiate the kernel at the opposite boundary.  A second use
of \eqref{opt:heat-poisson} gives
\begin{align}
 H(1,w)=\frac14\int_{0<u<v<1}
 &\bigl[b(u)a(v)A_l(u,0)C_{lh}(v-u)A_h(1-v,w)
 \nonumber\\[-2pt]
 &\quad+a(u)b(v)A_h(u,0)C_{hl}(v-u)A_l(1-v,w)\bigr]
 \,du\,dv.
 \label{opt:mixed-shape-derivative}
\end{align}
Each variation acts on a different endpoint; there is no term involving
a second displacement of the same boundary.  Each use of the Poisson
formula contributes one factor $1/2$.

It remains to establish that $u,v$ in
\eqref{opt:mixed-shape-derivative} are the times of the global extrema.
The survival event for the moving boundaries is
\[
 m_\eta:=\min_t\{W_t+\eta b(t)\}>l,
 \qquad M_\epsilon:=\max_t\{W_t-\epsilon a(t)\}<h.
\]
Almost sure uniqueness of the unperturbed extrema implies
\[
 \left.\frac{d}{d\eta}m_\eta\right|_0=b(T_-),
 \qquad
 \left.\frac{d}{d\epsilon}M_\epsilon\right|_0=-a(T_+).
\]
To justify differentiation of the event, first integrate its indicator
against a smooth, compactly supported test function $\psi(l,h)$ and
against a bounded test function of $W_1$.  The integration in the levels
is
\[
 J(m_\eta,M_\epsilon)
 =\int_{l<m_\eta}\int_{h>M_\epsilon}\psi(l,h)\,dh\,dl.
\]
Its mixed difference quotient is bounded by
$\|\psi\|_\infty\|a\|_\infty\|b\|_\infty$, because the two extrema
are Lipschitz in their respective parameters.  Since $J_{mM}=-\psi$,
the mixed derivative is
$a(T_+)b(T_-)\psi(\min W,\max W)$.  Dominated convergence is therefore
available before taking any level density.  Comparing this result with
\eqref{opt:mixed-shape-derivative}, for arbitrary level, endpoint, and
time test functions, identifies the joint measure.  Products $b(u)a(v)$
of smooth time test functions determine its restriction to $u<v$ and
yield \eqref{opt:ordered-extrema-density}.  The opposite order is the
second term.  The uniqueness and interior location of the extrema,
proved above, exclude a missing diagonal or time-endpoint component.
\end{proof}

\bibliographystyle{plain}
\bibliography{references}
\end{document}